\documentclass[reqno]{amsart}
\usepackage[colorlinks=true]{hyperref}
\usepackage{amsmath}
  \usepackage{paralist}
  \usepackage{amssymb}
 \usepackage{amsthm}

\AtBeginDocument{{\noindent\small \emph{}} \vspace{8mm}}
\begin{document}
\title[\hfil Studies of normalized solutions]
{Studies of normalized solutions to Schr\"{o}dinger equations with
Sobolev critical exponent and combined nonlinearities}

\author[X. F. Li]
{Xinfu Li}

  % in alphabetical order
\address{Xinfu Li \newline
School of Science, Tianjin University of Commerce, Tianjin 300134,
Peoples's Republic of China} \email{lxylxf@tjcu.edu.cn}
 % in alphabetical order
%\address{Yiqun Cheng \newline
%Department of Mathematics, Taiyuan University of Technology,
%Taiyuan, Shanxi 030024, P. R. China} \email{1906157258@qq.com}

\subjclass[2020]{35J20, 35B33.}
\keywords{multiplicity; orbital
stability, nonexistence; normalized solutions; critical exponent.}

\begin{abstract}
We consider the  Sobolev critical Schr\"{o}dinger equation with
combined nonlinearities
\begin{equation*}
\begin{cases}
-\Delta u=\lambda u+|u|^{2^*-2}u+\mu|u|^{q-2}u,\  \
x\in\mathbb{R}^{N},\\
u\in H^1(\mathbb{R}^N),\ \int_{\mathbb{R}^N}|u|^2dx=a,
\end{cases}
\end{equation*}
where $N\geq 3$, $\mu>0$, $\lambda\in \mathbb{R}$, $a>0$ and $q\in
(2,2^*)$. We prove in this paper

(1) Multiplicity and stability of solutions for $q\in
(2,2+\frac{4}{N})$ and $\mu a^{\frac{q(1-\gamma_q)}{2}}\leq
(2K)^{\frac{q\gamma_q-2^*}{2^*-2}}$ with
$\gamma_q:=\frac{N}{2}-\frac{N}{q}$ and  $K$ being some positive
constant. This result extends the results obtained in Jeanjean et
al. \cite{JEANJEAN-JENDREJ} and Jeanjean and Le \cite{Jeanjean-Le}
for the case $\mu
a^{\frac{q(1-\gamma_q)}{2}}<(2K)^{\frac{q\gamma_q-2^*}{2^*-2}}$ to
the case $\mu a^{\frac{q(1-\gamma_q)}{2}}\leq
(2K)^{\frac{q\gamma_q-2^*}{2^*-2}}$.

(2) Nonexistence of ground states for $q=2+\frac{4}{N}$ and  $\mu
a^{\frac{q(1-\gamma_q)}{2}}\geq\bar{a}_N$ with $\bar{a}_N$ being
some positive constant. We give  a new  proof to this result
different with Wei and Wu \cite{Wei-Wu 2021}.
\end{abstract}

\maketitle \numberwithin{equation}{section}
\newtheorem{theorem}{Theorem}[section]
\newtheorem{lemma}[theorem]{Lemma}
\newtheorem{definition}[theorem]{Definition}
\newtheorem{remark}[theorem]{Remark}
\newtheorem{proposition}[theorem]{Proposition}
\newtheorem{corollary}[theorem]{Corollary}
\allowdisplaybreaks

\section{Introduction and main results}

\setcounter{section}{1}
\setcounter{equation}{0}

In this paper, we study the  standing waves to the Sobolev critical
Schr\"{o}dinger equation with combined nonlinearities
\begin{equation}\label{e1.2}
i\partial_t\psi+\Delta
\psi+|\psi|^{2^*-2}\psi+\mu|\psi|^{q-2}\psi=0,\ (t,x)\in
\mathbb{R}\times \mathbb{R}^N,
\end{equation}
where $N\geq 3$, $\mu>0$, $2^*:=\frac{2N}{N-2}$ and $q\in(2,2^*)$.
Starting from the fundamental contribution by  Tao, Visan and Zhang
\cite{Tao-Visan-Zhang}, the nonlinear Schr\"{o}dinger equation with
combined nonlinearities attracted much attention, see for example
\cite{{Akahori-Ibrahim-Kikuchi},{Feng 2018},{Le Coz-Martel
2016},{Li-Zhao 2020},{Zhang 2006}}.

Standing waves to (\ref{e1.2}) are solutions of the form $\psi(t, x)
=e^{-i\lambda t}u(x)$, where $\lambda\in \mathbb{R}$ and
$u:\mathbb{R}^N\to \mathbb{C}$. Then $u$ satisfies the equation
\begin{equation}\label{e1.1}
-\Delta u=\lambda u+|u|^{2^*-2}u+\mu|u|^{q-2}u,\ x\in
\mathbb{R}^{N}.
\end{equation}
When looking for solutions to (\ref{e1.1}) one choice is to fix
$\lambda<0$ and to search for solutions to (\ref{e1.1}) as critical
points of the action functional
\begin{equation*}
J(u):=\int_{\mathbb{R}^N}\left(\frac{1}{2}|\nabla
u|^2-\frac{\lambda}{2}|u|^2-\frac{1}{2^*}|u|^{2^*}-\frac{\mu}{q}|u|^q\right)dx,
\end{equation*}
see for example \cite{{Alves-Souto},{Liu-Liao-Tang 2017}} and the
references therein. Another choice is to fix the $L^2$-norm of the
unknown $u$, that is, to consider the problem
\begin{equation}\label{e1.11}
\begin{cases}
-\Delta u=\lambda u+|u|^{2^*-2}u+\mu|u|^{q-2}u, \  x\in
\mathbb{R}^{N},\\
u\in H^1(\mathbb{R}^N),\ \int_{\mathbb{R}^N}|u|^2dx=a
\end{cases}
\end{equation}
with  fixed $a>0$ and unknown $\lambda\in \mathbb{R}$. In this
direction, define on $H^1(\mathbb{R}^N)$ the energy functional
\begin{equation*}
E(u):=\frac{1}{2}\int_{\mathbb{R}^N}|\nabla
u|^2dx-\frac{1}{2^*}\int_{\mathbb{R}^N}|u|^{2^*}dx
-\frac{\mu}{q}\int_{\mathbb{R}^N}|u|^{q}dx.
\end{equation*}
It is standard to check that $E\in C^1$ and a critical point of $E$
constrained to
\begin{equation}\label{e1.3}
S_a:=\{u\in H^1(\mathbb{R}^N):\int_{\mathbb{R}^N}|u|^2dx=a\}
\end{equation}
gives rise to a solution to (\ref{e1.11}). Such solution is usually
called a normalized solution of (\ref{e1.1}) on $S_a$, which is the
aim of this paper. In studying  normalized solutions to the
Schr\"{o}dinger equation
\begin{equation}\label{e1.12}
-\Delta u=\lambda u+|u|^{p-2}u+\mu|u|^{q-2}u,\ x\in \mathbb{R}^{N}
\end{equation}
with $2<q<p\leq 2^*$, the so-called $L^2$-critical exponent
$2+\frac{4}{N}$ plays an important role.  A very complete analysis
of the various cases that may happen for (\ref{e1.12})-(\ref{e1.3}),
depending on the values of $(p,q)$, has been provided recently in
\cite{{JEANJEAN-JENDREJ},{Jeanjean-Le},{Soave JDE},{Soave
JFA},{Wei-Wu 2021}}. For future reference, we recall

\begin{definition}\label{def1.1}
We say that $u$ is a normalized ground state to (\ref{e1.1}) on
$S_a$ or a ground state to (\ref{e1.11}) if
\begin{equation*}
E(u)=c_a^g:=\inf\{E(v): v\in S_a,\ (E|_{S_a})'(v)=0\}.
\end{equation*}
The set of the ground states to (\ref{e1.11}) will be denoted by
$\mathcal{G}_a$.
\end{definition}

\begin{definition}\label{def1.2}
$\mathcal{G}_a$ is orbitally stable if for every $\epsilon>0$ there
exists $\delta>0$ such that, for any $\psi_0\in H^1(\mathbb{R}^N)$
with $\inf_{v\in \mathcal{G}_a}\|\psi_0-v\|_{H^1}<\delta$, we have
\begin{equation*}
\inf_{v\in \mathcal{G}_a}\|\psi(t,\cdot)-v\|_{H^1}<\epsilon\ \
\mathrm{for\ any\ }t>0,
\end{equation*}
where $\psi(t,x)$ denotes the solution to (\ref{e1.2}) with initial
value $\psi_0$.

A standing wave $e^{-i\lambda t}u$ is strongly unstable if for every
$\epsilon>0$ there exists $\psi_0\in H^1(\mathbb{R}^N)$ such that
$\|\psi_0-u\|_{H^1}<\epsilon$, and $\psi(t,x)$ blows up in finite
time.
\end{definition}

We first consider the case  $q\in (2,2+\frac{4}{N})$. Recently,
\cite{JEANJEAN-JENDREJ} and \cite{Jeanjean-Le}  studied the
multiplicity and stability of solutions to (\ref{e1.11}) under the
condition $\mu
a^{\frac{q(1-\gamma_q)}{2}}<(2K)^{\frac{q\gamma_q-2^*}{2^*-2}}$
which is obtained as follows. By using  the Sobolev inequality (see
\cite{Brezis-Nirenberg 1983})
\begin{equation}\label{e1.13}
S\|u\|_{2^*}^2\leq \|\nabla u\|_2^2\ \  \mathrm{for\ all}\ u\in
D^{1,2}(\mathbb{R}^N)
\end{equation}
and the Gagliardo-Nirenberg inequality (see \cite{Weinstein 1983})
\begin{equation}\label{e1.14}
\|u\|_{q}\leq C_{N,q} \|\nabla u\|_2^{\gamma_q}\|u\|_2^{1-\gamma_q}\
\  \mathrm{for\ all}\ u\in H^{1}(\mathbb{R}^N),
\gamma_q:=\frac{N}{2}-\frac{N}{q},
\end{equation}
we have for any $u\in S_a$,
\begin{equation*}
\begin{split}
E(u)&\geq \frac{1}{2}\|\nabla
u\|_2^2-\frac{1}{2^*}\left(S^{-1}\|\nabla
u\|_2^2\right)^{\frac{2^*}{2}}-\frac{\mu}{q}C_{N,q}^q\|\nabla
u\|_2^{q\gamma_q}\|u\|_2^{q(1-\gamma_q)}\\
&=\|\nabla u\|_2^2f_{\mu, a}(\|\nabla u\|_2^2),
\end{split}
\end{equation*}
where
\begin{equation*}
f_{\mu,
a}(\rho):=\frac{1}{2}-\frac{1}{2^*}S^{-\frac{2^*}{2}}\rho^{\frac{2^*}{2}-1}
-\frac{\mu}{q}C_{N,q}^qa^{\frac{q(1-\gamma_q)}{2}}\rho^{\frac{q\gamma_q}{2}-1},\
\rho\in (0,\infty).
\end{equation*}
Direct calculations give that
\begin{equation*}%\label{e1.15}
\max_{\rho>0}f_{\mu,a}(\rho)=f_{\mu,a}(\rho_{\mu,a})
\left\{\begin{array}{ll}
>0,&  \ \mathrm{if}\ \mu a^{\frac{q(1-\gamma_q)}{2}}<(2K)^{\frac{q\gamma_q-2^*}{2^*-2}},\\
=0,&  \ \mathrm{if}\ \mu a^{\frac{q(1-\gamma_q)}{2}}=(2K)^{\frac{q\gamma_q-2^*}{2^*-2}},\\
<0,&  \ \mathrm{if}\ \mu
a^{\frac{q(1-\gamma_q)}{2}}>(2K)^{\frac{q\gamma_q-2^*}{2^*-2}},
\end{array}\right.
\end{equation*}
where
\begin{equation*}
\rho_{\mu,a}=\left(\frac{(2-q\gamma_q)2^*S^{\frac{2^*}{2}}C_{N,q}^q\mu
a^{\frac{q(1-\gamma_q)}{2}}}{q(2^*-2)}\right)^{\frac{2}{2^*-q\gamma_q}}
\end{equation*}
and
\begin{equation*}
K=\frac{2^*-q\gamma_q}{2^*(2-q\gamma_q)S^{\frac{2^*}{2}}}
\left(\frac{2^*S^{\frac{2^*}{2}}(2-q\gamma_q)C_{N,q}^q}{q(2^*-2)}\right)^{\frac{2^*-2}{2^*-q\gamma_q}}.
\end{equation*}
Thus, the domain $\{(\mu,a)\in \mathbb{R}^2: \mu>0,\ a>0\}$ is
divided into three parts $\Omega_1$, $\Omega_2$ and $\Omega_3$ by
the curve $\mu
a^{\frac{q(1-\gamma_q)}{2}}=(2K)^{\frac{q\gamma_q-2^*}{2^*-2}}$ with
\begin{equation*}
\Omega_1=\{(\mu,a)\in \mathbb{R}^2: \mu>0,\ a>0, \ \mu
a^{\frac{q(1-\gamma_q)}{2}}<(2K)^{\frac{q\gamma_q-2^*}{2^*-2}}\},
\end{equation*}
\begin{equation*}
\Omega_2=\{(\mu,a)\in \mathbb{R}^2: \mu>0,\ a>0, \ \mu
a^{\frac{q(1-\gamma_q)}{2}}=(2K)^{\frac{q\gamma_q-2^*}{2^*-2}}\}
\end{equation*}
and
\begin{equation*}
\Omega_3=\{(\mu,a)\in \mathbb{R}^2: \mu>0,\ a>0, \ \mu
a^{\frac{q(1-\gamma_q)}{2}}>(2K)^{\frac{q\gamma_q-2^*}{2^*-2}}\}.
\end{equation*}
For fixed $\mu>0$, define $a_0$ such that
\begin{equation*}%\label{e6.3}
\mu
a_0^{\frac{q(1-\gamma_q)}{2}}=(2K)^{\frac{q\gamma_q-2^*}{2^*-2}}.
\end{equation*}
Then $\Omega_1$ can also be expressed as
$\Omega_1=\{(\mu,a)\in\mathbb{R}^2,\mu>0, 0<a<a_0\}$. Define
\begin{equation*}
\rho_0:=\rho_{\mu,a_0}=\left(\frac{2^*(2-q\gamma_q)S^{\frac{2^*}{2}}}{2(2^*-q\gamma_q)}\right)^{\frac{2}{2^*-2}},
\end{equation*}
\begin{equation*}
B_{\rho_0}:=\{u\in H^1(\mathbb{R}^N):\|\nabla u\|_2^2<\rho_0\},\ \ \
V_a:=S_a\cap B_{\rho_0}.
\end{equation*}
Under the condition $(\mu,a)\in \Omega_1$, it was proved in
(\cite{JEANJEAN-JENDREJ}, Lemma 2.4) that
\begin{equation*}
m_a:=\inf_{u\in V_a}E(u)<0\leq \inf_{u\in \partial V_a}E(u).
\end{equation*}
Then by using the concentration compactness principle,
\cite{JEANJEAN-JENDREJ} obtained the following results:\\
\textbf{Theorem A}. Let $N\geq 3$, $q\in(2,2+\frac{4}{N})$, $\mu>0$
and $0<a<a_0$. Then (\ref{e1.11}) has a ground state $u$, which is a
 minimizer of $E$ on $V_a$. In addition, any ground state of
(\ref{e1.11}) is a  minimizer of $E$ on $V_a$. Moreover,
$\mathcal{G}_a$ is compact, up to translation, and it is orbitally
stable.

\medskip

Noting that
\begin{equation*}
\begin{split}
\Omega_1\cup \Omega_2&=\{(\mu,a)\in\mathbb{R}^2,\mu>0, 0<a\leq
a_0\}\\
&=\{(\mu,a)\in \mathbb{R}^2: \mu>0,\ a>0, \ \mu
a^{\frac{q(1-\gamma_q)}{2}}\leq
(2K)^{\frac{q\gamma_q-2^*}{2^*-2}}\},
\end{split}
\end{equation*}
we find that the results in  Theorem  A  can be extended to the case
$(\mu,a)\in \Omega_1\cup \Omega_2$ by repeating word by word  the
proof of the paper \cite{JEANJEAN-JENDREJ}:

\begin{theorem}\label{thm1.1}
Let $N\geq 3$, $q\in(2,2+\frac{4}{N})$, $\mu>0$,  $a>0$ and $\mu
a^{\frac{q(1-\gamma_q)}{2}}\leq (2K)^{\frac{q\gamma_q-2^*}{2^*-2}}$.
Then (\ref{e1.11}) has a ground state $u$, which is a  minimizer of
$E$ on $V_a$. In addition, any ground state of (\ref{e1.11}) is a
minimizer of $E$ on $V_a$. Moreover, $\mathcal{G}_a$ is compact, up
to translation, and it is orbitally stable.
\end{theorem}

In addition, under the condition $q\in (2,2+\frac{4}{N})$ and
$(\mu,a)\in \Omega_1$, for any $u\in S_a$, the fiber map
\begin{equation}\label{e6.6}
\begin{split}
\Psi_u(\tau):=E(u_\tau)=\frac{1}{2}\tau^{2}\|\nabla
u\|_2^2-\frac{1}{2^*}\tau^{2^*}\|u\|_{2^*}^{2^*}-\frac{\mu}{q}\tau^{q\gamma_q}\|u\|_q^{q}
\end{split}
\end{equation}
with
\begin{equation*}%\label{e6.5}
u_\tau(x):=\tau^{\frac{N}{2}}u(\tau x),\ x\in\mathbb{R}^N,\ \tau>0
\end{equation*}
has exactly two critical points, one is a local minimum point at
negative level and the other one is a global maximum point at
positive level. Consequently, the Poho\v{z}aev set
\begin{equation*}
\mathcal{P}_{a}:=\{u\in S_a: P(u)=0\}\ \mathrm{with}\ P(u):=\|\nabla
u\|_2^2-\|u\|_{2^*}^{2^*}-\mu \gamma_q\|u\|_q^{q}
\end{equation*}
admits the decomposition into the disjoint union
$\mathcal{P}_{a}=\mathcal{P}_{a,+}\cup \mathcal{P}_{a,-}$, where
\begin{equation*}
\mathcal{P}_{a,+}:=\{u\in \mathcal{P}_{a}: E(u)<0\}\ \mathrm{and} \
\mathcal{P}_{a,-}:=\{u\in \mathcal{P}_{a}: E(u)> 0\},
\end{equation*}
see Lemma 2.4 in \cite{Jeanjean-Le}. Moreover, $0<\inf_{u\in
\mathcal{P}_{a,-}}E(u)<\frac{1}{N}{S^{\frac{N}{2}}}$, see
Propositions 1.15 and 1.16 in \cite{Jeanjean-Le} for $N\geq 4$ and
Lemma 3.1 and Remark 3.1 in \cite{Wei-Wu 2021} for $N\geq 3$. By
using these results and the mountain pass lemma, \cite{Jeanjean-Le}
obtained the following
results (The case $N=3$ was complemented  by \cite{Wei-Wu 2021}):\\
\textbf{Theorem B}. Let $N\geq 3$, $q\in(2,2+\frac{4}{N})$, $\mu>0$,
$a>0$ and $\mu
a^{\frac{q(1-\gamma_q)}{2}}<(2K)^{\frac{q\gamma_q-2^*}{2^*-2}}$.
Then there exists a mountain pass type solution $v$ to (\ref{e1.11})
with $\lambda<0$ and
\begin{equation*}
0<E(v)=\inf_{u\in
\mathcal{P}_{a,-}}E(u)<m_a+\frac{S^{\frac{N}{2}}}{N}.
\end{equation*}
Moreover,
the associated standing wave $e^{-i\lambda t}v(x)$ is strongly
unstable.

\medskip

We wonder what will happen if $\mu
a^{\frac{q(1-\gamma_q)}{2}}=(2K)^{\frac{q\gamma_q-2^*}{2^*-2}}$. By
examining the proof of Theorem B, we find that $\inf_{u\in
\mathcal{P}_{a,-}}E(u)\geq 0$ and there are two possibilities:

(1) If $\inf_{u\in \mathcal{P}_{a,-}}E(u)=0$, then by using
(\ref{e1.13}) and  (\ref{e1.14}), we can show that there is no $u\in
\mathcal{P}_{a}$ such that $E(u)=0$;

(2) If $\inf_{u\in \mathcal{P}_{a,-}}E(u)>0$, we can certainly
obtain the same results as in Theorem B.\\
At first, we try to show that $\inf_{u\in \mathcal{P}_{a,-}}E(u)=0$
by choosing some functions and find that it is a difficult task.
Then we turn to study the properties satisfied by  $\{u_n\}$ with
$\{u_n\}\subset \mathcal{P}_{a,-}$ and $E(u_n)\to 0$, and find that
such $\{u_n\}$ does not exist. Hence, $\inf_{u\in
\mathcal{P}_{a,-}}E(u)>0$ (see Lemma \ref{lem7.21}) and  we obtain
the following result:

\begin{theorem}\label{thm1.8}
Let $N\geq 3$, $q\in(2,2+\frac{4}{N})$, $\mu>0$, $a>0$ and $\mu
a^{\frac{q(1-\gamma_q)}{2}}\leq (2K)^{\frac{q\gamma_q-2^*}{2^*-2}}$.
Then there exists a mountain pass type solution $v$ to (\ref{e1.11})
with $\lambda<0$ and
\begin{equation*}
0<E(v)=\inf_{u\in
\mathcal{P}_{a,-}}E(u)<m_a+\frac{1}{N}S^{\frac{N}{2}}.
\end{equation*}
Moreover, the associated standing wave $e^{-i\lambda t}v(x)$ is
strongly unstable.
\end{theorem}

Now we consider the case $q=2+\frac{4}{N}$. In this case,
\cite{Soave JFA} recently obtained that (\ref{e1.11}) admits a
ground state if $\mu a^{\frac{q(1-\gamma_q)}{2}}<\bar{a}_N$, where
$\bar{a}_N:=\frac{q}{2C_{N,q}^q}$ and $C_{N,q}$ is defined in
\ref{e1.14},
 while if $\mu a^{\frac{q(1-\gamma_q)}{2}}\geq \bar{a}_N$,
\cite{Wei-Wu 2021} obtained that (\ref{e1.11}) has no ground states
by showing that $c^{po}:=\inf_{v\in \mathcal{P}_a}E(v)=0$.
Precisely, \cite{Wei-Wu 2021} obtained:

\begin{theorem}\label{thm1.9}
Let $N\geq 3$, $q=2+\frac{4}{N}$, $\mu>0$, $a>0$ and $\mu
a^{\frac{q(1-\gamma_q)}{2}}\geq \bar{a}_N$. Then $c^{po}=0$.
Moreover, $c^{po}$ can not be attained and (\ref{e1.11}) has no
ground states.
\end{theorem}

In the proof of Theorem \ref{thm1.9}, a key step is to show that
$c^{po}=0$. In \cite{Wei-Wu 2021}, the proof of $c^{po}=0$ depends
on  the monotonicity of $c^{po}$ with $\mu$ (see Lemma 3.3 in
\cite{Wei-Wu 2021}). In this paper, we will give a different proof
by choosing appropriate functions but without using the monotonicity
of $c^{po}$ (see Lemma \ref{lem3.2}). Our methods used to prove
$c^{po}=0$ has some similarities with the methods used to prove
$\inf_{u\in \mathcal{P}_{a,-}}E(u)>0$ in Theorem \ref{thm1.8} (see
Lemmas \ref{lem7.21} and \ref{lem3.2}). It seems to be two sides of
a question.

\medskip

This paper is organized as follows. In Section 2, we  give the proof
of Theorem \ref{thm1.8}. Section 3 is devoted to the proof of
Theorem \ref{thm1.9}.

\medskip

\textbf{Notation}: In this paper, it is understood that all
functions, unless otherwise stated, are complex valued, but for
simplicity we write $L^r(\mathbb{R}^N)$, $H^1(\mathbb{R}^N)$,
$D^{1,2}(\mathbb{R}^N)$. For $1\leq r<\infty$, $L^r(\mathbb{R}^N)$
is the usual Lebesgue space endowed with the norm
$\|u\|_r^r:=\int_{\mathbb{R}^N}|u|^rdx$, $H^{1}(\mathbb{R}^N)$
 is the usual Sobolev space endowed with the
norm $\|u\|_{H^{1}}^2:=\|\nabla u\|_2^2+\|u\|_2^2$, and
$D^{1,2}(\mathbb{R}^N):=\left\{u\in L^{2^*}(\mathbb{R}^N):|\nabla
u|\in L^{2}(\mathbb{R}^N)\right\}$. $H_{r}^1(\mathbb{R}^N)$ denotes
the subspace of functions in $H^1(\mathbb{R}^N)$ which are radially
symmetric with respect to zero.  $S_{a,r}:=S_a\cap
H_{r}^1(\mathbb{R}^N)$.

\section{Proof of Theorem \ref{thm1.8}}
\setcounter{section}{2} \setcounter{equation}{0}

We use the strategy of \cite{Jeanjean-Le} to prove Theorem
\ref{thm1.8}. Since many results obtained in \cite{Jeanjean-Le} can
be extended to the case $\mu a^{\frac{q(1-\gamma_q)}{2}}\leq
(2K)^{\frac{q\gamma_q-2^*}{2^*-2}}$ directly, we just list them
without proofs and concentrate on proving the new result (Lemma
\ref{lem7.21}).

\begin{lemma}\label{lem6.6}
(Lemma 2.4, \cite{Jeanjean-Le}) Let the assumptions in Theorem
\ref{thm1.8} hold. Then for every $u\in S_a$, the function
$\Psi_u(\tau)$ defined in (\ref{e6.6}) has exactly two critical
points $\tau_u^+$ and $\tau_u^-$ with $0<\tau_u^+<\tau_u^-$.
Moreover:

(1) $\tau_u^+$ is a local minimum point for $\Psi_u(\tau)$,
$E(u_{\tau_u^+})<0$ and $u_{\tau_u^+}\in V_a$.

(2) $\tau_u^-$ is a global maximum point for $\Psi_u(\tau)$,
$\Psi_u'(\tau)<0$ for $\tau>\tau_u^-$ and
\begin{equation*}
E(u_{\tau_u^-})\geq \inf_{u\in \partial V_a}E(u)\geq 0.
\end{equation*}
In particular, if $\mu a^{\frac{q(1-\gamma_q)}{2}}<
(2K)^{\frac{q\gamma_q-2^*}{2^*-2}}$, then $\inf_{u\in \partial
V_a}E(u)> 0$.

(3) $\Psi_u''(\tau_u^-)<0$ and the maps $u\in S_a\mapsto \tau_u^-\in
\mathbb{R}$ is of class $C^1$.
\end{lemma}

\begin{lemma}\label{lem6.4}
(Lemma 2.6, \cite{JEANJEAN-JENDREJ}) Let  $N\geq 3$,
$q\in(2,2+\frac{4}{N})$, $\mu>0$, $a\in (0,a_0]$. Then for any
$a_1\in (0, a)$, we have $ m_a\leq m_{a_1}+m_{a-a_1}$, and if
$m_{a_1}$ or $m_{a-a_1}$ is reached then the inequality is strict.
\end{lemma}

Now we set
\begin{equation*}
M_{r}(a):=\inf_{g\in \Gamma_{r}(a)}\max_{t\in [0,\infty)}E(g(t)),
\end{equation*}
where
\begin{equation*}
\Gamma_{r}(a):=\{g\in C([0,\infty),S_{a,r}):g(0)\in
\mathcal{P}_{a,+}, \exists t_{g}\ s.t.\ g(t)\in E_{2m_{a}}\
\mathrm{for}\ t\geq t_{g}\}
\end{equation*}
with $E_{c}:=\{u\in H^1(\mathbb{R}^N): E(u)<c\}$. Then we have

\begin{lemma}\label{pro1.2}
(Proposition 1.10, \cite{Jeanjean-Le})  Let the assumptions in
Theorem \ref{thm1.8} hold. Then there exists a Palais-Smale sequence
$\{u_n\}\subset S_{a,r}$ for $E|_{S_a}$ at level $M_{r}(a)$, with
$P(u_n)\to 0$ as $n\to \infty$.
\end{lemma}

Next we study the value of  $M_r(a)$. For this aim, we set
\begin{equation*}
M(a):=\inf_{g\in \Gamma(a)}\max_{t\in [0,\infty)}E(g(t)),
\end{equation*}
where
\begin{equation*}%\label{e6.8}
\Gamma(a):=\{g\in C([0,\infty),S_{a}):g(0)\in V_a\cap E_{0}, \exists
t_{g}\ s.t.\ g(t)\in E_{2m_{a}}, t\geq t_{g}\}.
\end{equation*}

\begin{lemma}\label{pro1.3}
(Proposition 1.15 and Remark 5.1, \cite{Jeanjean-Le})  Let the
assumptions in Theorem \ref{thm1.8} hold. Then
\begin{equation*}
M_{r}(a)=M(a)=\inf_{\mathcal{P}_{a,-}}E(u)=\inf_{\mathcal{P}_{a,-}\cap
H_r^1(\mathbb{R}^N)}E(u).
\end{equation*}
\end{lemma}

\begin{lemma}\label{lem7.21}
Let the assumptions in Theorem \ref{thm1.8} hold. Then $\inf_{u\in
\mathcal{P}_{a,-}}E(u)>0$.
\end{lemma}

\begin{proof}
In view of  Lemma \ref{lem6.6}, we just need to show that
$\inf_{u\in \mathcal{P}_{a,-}}E(u)\neq 0$ for the case $\mu
a^{\frac{q(1-\gamma_q)}{2}}=(2K)^{\frac{q\gamma_q-2^*}{2^*-2}}$.
Suppose by contradiction that $\inf_{u\in \mathcal{P}_{a,-}}E(u)=0$.
By Lemma \ref{pro1.3},
$\inf_{\mathcal{P}_{a,-}}E(u)=\inf_{\mathcal{P}_{a,-}\cap
H_{r}^1(\mathbb{R}^N)}E(u)$, which combined with Lemma \ref{lem6.6}
yields that there exists $\{u_n\}\subset S_a\cap
H_r^1(\mathbb{R}^N)$ such that $P(u_n)=0$ and $E(u_n)=A_n$, where
$A_n\geq 0$ and $A_n\to 0$ as $n\to \infty$. By using $E(u_n)=A_n$,
$P(u_n)=0$, $\|u_n\|_2^2=a$, (\ref{e1.13}) and (\ref{e1.14}), we
obtain that
\begin{equation}\label{e7.22}
\begin{cases}
\|\nabla u_n\|_2^2=\frac{2(2^*-q\gamma_q)}{q(2^*-2)}\mu
\|u_n\|_q^q+C_1A_n\leq \frac{2(2^*-q\gamma_q)}{q(2^*-2)}\mu
C_{N,q}^q a^{\frac{q(1-\gamma_q)}{2}}\|\nabla u_n\|_2^{q\gamma_q}+C_1A_n,\\
\|\nabla
u_n\|_2^2=\frac{2(2^*-q\gamma_q)}{2^*(2-q\gamma_q)}\|u_n\|_{2^*}^{2^*}-C_2A_n
\leq \frac{2(2^*-q\gamma_q)}{2^*(2-q\gamma_q)}
S^{-\frac{2^*}{2}}\|\nabla u_n\|_2^{2^*}-C_2A_n,
\end{cases}
\end{equation}
where $C_1$ and $C_2$ are some positive constants. Consequently,
$\liminf_{n\to\infty}\|\nabla u_n\|_2^2>0$ and
\begin{equation*}
\begin{cases}
\|\nabla u_n\|_2^{2-q\gamma_q}\leq
\frac{2(2^*-q\gamma_q)}{q(2^*-2)}\mu
C_{N,q}^q a^{\frac{q(1-\gamma_q)}{2}}+o_n(1)
=\left(\frac{2^*(2-q\gamma_q)}{2(2^*-q\gamma_q)} S^{\frac{2^*}{2}}\right)^{\frac{2-q\gamma_q}{2^*-2}}+o_n(1),\\
\|\nabla u_n\|_2^{2^*-2}\geq
\frac{2^*(2-q\gamma_q)}{2(2^*-q\gamma_q)} S^{\frac{2^*}{2}}+o_n(1),
\end{cases}
\end{equation*}
which implies that
\begin{equation*}
\begin{cases}
\|\nabla u_n\|_2^{2}\leq  \rho_0+o_n(1),\\
\|\nabla u_n\|_2^{2}\geq \rho_0+o_n(1).
\end{cases}
\end{equation*}
Hence, $\|\nabla u_n\|_2^2\to \rho_0$ as $n\to \infty$, which
combined with (\ref{e7.22}) gives that
\begin{equation*}%\label{e7.23}
\begin{cases}
\|u_n\|_q^q\to C_{N,q}^q \|u_n\|_2^{q(1-\gamma_q)}\|\nabla u_n\|_2^{q\gamma_q},\\
\|u_n\|_{2^*}^{2^*}\to S^{-\frac{2^*}{2}}\|\nabla u_n\|_2^{2^*}
\end{cases}
\end{equation*}
as $n\to \infty$. That is, $\{u_n\}\subset H_r^1(\mathbb{R}^N)$ is a
minimizing sequence of
\begin{equation}\label{e7.24}
\begin{split}
\frac{1}{C_{N,q}^q}:&=\inf_{ u\in
H^{1}(\mathbb{R}^N)\setminus\{0\}}\frac{\|\nabla
u\|_2^{q\gamma_q}\|u\|_2^{q(1-\gamma_q)}}{\|u\|_q^q}
\end{split}
\end{equation}
and
\begin{equation}\label{e7.25}
\begin{split}
S:&=\inf_{ u\in D^{1,2}(\mathbb{R}^N)\setminus\{0\}}\frac{\|\nabla
u\|_2^2}{\|u\|_{2^*}^2}.
\end{split}
\end{equation}
Since $\{u_n\}\subset H_r^1(\mathbb{R}^N)$ is bounded, there exists
$u_0\in H_r^1(\mathbb{R}^N)\backslash \{0\}$ such that
$u_n\rightharpoonup u_0$ weakly in $H^1(\mathbb{R}^N)$, $u_n\to u_0$
strongly in $L^t(\mathbb{R}^N)$ with $t\in (2,2^*)$ and $u_n\to u_0$
a.e. in $\mathbb{R}^N$. By the weak convergence, we have
$\|u_0\|_2^2\leq \|u_n\|_2^2$ and $\|\nabla u_0\|_2^2\leq \|\nabla
u_n\|_2^2$. Consequently, $u_0$ is a minimizer of (\ref{e7.24}) and
$u_n\to u_0$ strongly in $H^1(\mathbb{R}^N)$. By Theorem B in
\cite{Weinstein 1983}, $u_0$ is the ground state of the equation
\begin{equation}\label{e7.26}
\frac{(q-2)N}{4}\Delta
u-\left(1+\frac{(q-2)(2-N)}{4}\right)u+|u|^{q-2}u=0.
\end{equation}
By using (\ref{e7.25}) and $u_n\to u_0$ strongly in
$H^1(\mathbb{R}^N)$, we obtain that $u_0$ is a minimizer of $S$. So
$u_0$ is of the form
\begin{equation}\label{e7.27}
u_0=C\left(\frac{b}{b^2+|x|^2}\right)^{\frac{N-2}{2}},
\end{equation}
where $C>0$ is a fixed constant and $b\in (0,\infty)$ is a
parameter, see \cite{Brezis-Nirenberg 1983}. (\ref{e7.27})
contradicts to (\ref{e7.26}). Thus, $\inf_{u\in
\mathcal{P}_{a,-}}E(u)>0$.
\end{proof}

\begin{lemma}\label{lem7.1}
(Lemma 3.1 and Remark 3.1, \cite{Wei-Wu 2021}) Let the assumptions
in Theorem \ref{thm1.8} hold. Then
\begin{equation*}
\inf_{u\in \mathcal{P}_{a,-}}E(u)<m_a+\frac{1}{N}S^{\frac{N}{2}}.
\end{equation*}
\end{lemma}

\begin{lemma}\label{pro1.4}
(Proposition 1.11, \cite{Jeanjean-Le}) Assume the assumptions in
Theorem \ref{thm1.8} hold. Let $\{u_n\}\subset S_{a,r}$ be a
Palais-Smale sequence for $E|_{S_a}$ at level $c$, with $P(u_n)\to
0$ as $n\to \infty$. If
\begin{equation*}
0<c<m_a+\frac{1}{N}S^{\frac{N}{2}},
\end{equation*}
then up a subsequence, $u_n\to u$ strongly in $H^1(\mathbb{R}^N)$,
and $u$ is a radial solution to (\ref{e1.11}) with $E(u)=c$ and some
$\lambda<0$.
\end{lemma}

\textbf{Proof of Theorem \ref{thm1.8}}. The existence of a mountain
pass type solution  is a direct result of Lemmas \ref{pro1.2},
\ref{pro1.3}, \ref{lem7.21}, \ref{lem7.1} and \ref{pro1.4}. The
strong instability  of the associated standing wave is the same as
Theorem 1.9 in \cite{Jeanjean-Le}.

\section{Proof of Theorem \ref{thm1.9}}
\setcounter{section}{3} \setcounter{equation}{0}

Let $\Psi_u(\tau)$ be defined in (\ref{e6.6}) and define
\begin{equation*}
\Phi_u(\tau):=P(u_\tau)=\tau^{2}\|\nabla
u\|_2^2-\tau^{2^*}\|u\|_{2^*}^{2^*}-\mu\gamma_q\tau^{q\gamma_q}\|u\|_q^{q}.
\end{equation*}

\begin{lemma}\label{lem3.1}
Let $N\geq 3$, $q=2+\frac{4}{N}$, $\mu>0$, $a>0$ and $\mu
a^{\frac{q(1-\gamma_q)}{2}}\geq \bar{a}_N$. Then we have the
following results:

(1) If $u\in S_a$ such that $\|\nabla u\|_2^2>\mu
\gamma_q\|u\|_q^q$, then there exists a unique $\tau_u\in
(0,\infty)$ such that $P(u_{\tau_u})=0$.  $\tau_u$ is the unique
critical point of $\Psi_u(\tau)$, and is a maximum point at positive
level. Moreover,  $P(u)\leq 0 \Leftrightarrow \tau_u\leq 1$.

(2) If $u\in S_a$ such that $\|\nabla u\|_2^2\leq \mu
\gamma_q\|u\|_q^q$, then there does not exist $\tau\in (0,\infty)$
such that $P(u_\tau)=0$, and $\Psi_u(\tau)$ is negative and is
strictly decreasing on $(0,\infty)$.
\end{lemma}

\begin{proof}
The proof of (1)  is similar to Lemma 6.2 in \cite{Soave JDE}, and
the proof of (2) is a direct result of the expressions of
$P(u_\tau)$ and $\Psi_u(\tau)$.
\end{proof}

\begin{lemma}\label{lem3.2}
Let $N\geq 3$, $q=2+\frac{4}{N}$, $\mu>0$, $a>0$  and $\mu
a^{\frac{q(1-\gamma_q)}{2}}\geq \bar{a}_N$. Then
\begin{equation*}
c^{po}:=\inf_{v\in\mathcal{P}_a}E(v)=0.
\end{equation*}
\end{lemma}

\begin{proof}
By Lemma \ref{lem3.1}, for any $u\in S_a$ with $\|\nabla u\|_2^2>\mu
\gamma_q\|u\|_q^q$,  there exists a unique $\tau_u\in (0,\infty)$
such that $P(u_{\tau_u})=0$. Hence, $\tau_u$ satisfies
\begin{equation*}
(\tau_u)^2(\|\nabla
u\|_2^2-\mu\gamma_q\|u\|_q^q)=(\tau_u)^{2^*}\|u\|_{2^*}^{2^*}
\end{equation*}
and
\begin{equation*}
E(u_{\tau_u})=\left(\frac{1}{2}-\frac{1}{2^*}\right)\left(\frac{\|\nabla
u\|_2^2-\mu\gamma_q\|u\|_q^q}{\|u\|_{2^*}^2}\right)^{\frac{2^*}{2^*-2}}.
\end{equation*}
Consequently, to prove Lemma \ref{lem3.2}, it is enough to find
$\{u_n\}\subset S_a$ with
\begin{equation}\label{e3.1}
\|\nabla u_n\|_2^2>\mu \gamma_q\|u_n\|_q^q\ \  \mathrm{and}\ \
\frac{\|\nabla u_n\|_2^2-\mu
\gamma_q\|u_n\|_q^q}{\|u_n\|_{2^*}^{2}}\to 0.
\end{equation}

\textbf{Case 1 ($\mu a^{\frac{q(1-\gamma_q)}{2}}= \bar{a}_N$)}.
Since in this case we have $\mu \gamma_q C_{N,q}^q
a^{\frac{q(1-\gamma_q)}{2}}=1$, so for any $u\in S_a$, we get that
\begin{equation*}
\mu \gamma_q\|u\|_q^q\leq \mu \gamma_q C_{N,q}^q
a^{\frac{q(1-\gamma_q)}{2}}\|\nabla u\|_2^2=\|\nabla u\|_2^2
\end{equation*}
and the equality holds if and only if $u$ is a minimizer of
\begin{equation}\label{e3.2}
\begin{split}
\frac{1}{C_{N,q}^q}:&=\inf_{ u\in
H^{1}(\mathbb{R}^N)\setminus\{0\}}\frac{\|\nabla
u\|_2^{q\gamma_q}\|u\|_2^{q(1-\gamma_q)}}{\|u\|_q^q},
\end{split}
\end{equation}
or equivalently, $u$ is the ground state of the equation
\begin{equation*}
\frac{(q-2)N}{4}\Delta
u-\left(1+\frac{(q-2)(2-N)}{4}\right)u+|u|^{q-2}u=0.
\end{equation*}

Now let $u\in S_a$ be a minimizer of (\ref{e3.2}) and $\varphi(x)
\in C_c^{\infty}(\mathbb{R}^N)$ be a cut off function satisfying:
(a) $0\leq \varphi(x)\leq 1$ for any $x\in \mathbb{R}^N$; (b)
$\varphi(x)\equiv 1$ in $B_1$; (c) $\varphi(x)\equiv 0$ in
$\mathbb{R}^N\setminus \overline{B_2}$. Here, $B_s$ denotes the ball
in $\mathbb{R}^N$ of center at origin and radius $s$. Define
\begin{equation*}
v_n(x):=\varphi(\frac{x}{n})u(x),\ \
u_n(x):=a^{1/2}\|v_n\|_2^{-1}v_n(x).
\end{equation*}
Then $v_n\to u$ strongly in $H^1(\mathbb{R}^N)$,
\begin{equation*}
\|u_n\|_2^2=a,\ \
\|u_n\|_{2^*}^{2^*}=(a^{1/2}\|v_n\|_2^{-1})^{2^*}\|v_n\|_{2^*}^{2^*},
\end{equation*}
\begin{equation*}
\|\nabla u_n\|_2^2=(a^{1/2}\|v_n\|_2^{-1})^2\|\nabla v_n\|_2^2,\ \
\|u_n\|_{q}^{q}=(a^{1/2}\|v_n\|_2^{-1})^q\|v_n\|_{q}^{q}.
\end{equation*}
Next we show that $\{u_n\}\subset S_a$ satisfies (\ref{e3.1}). Since
$u_n$ is not a minimizer of (\ref{e3.2}), we deduce that  $\|\nabla
u_n\|_2^2>\mu \gamma_q\|u_n\|_q^q$. Noting that $u\in S_a$ is a
minimizer of (\ref{e3.2}), by direct calculations, we obtain that
\begin{equation*}
\begin{split}
\frac{\|\nabla u_n\|_2^2-\mu \gamma_q\|u_n\|_q^q}{\|u_n\|_{2^*}^{2}}
&=\frac{(a^{1/2}\|v_n\|_2^{-1})^2\|\nabla v_n\|_2^2-\mu
\gamma_q(a^{1/2}\|v_n\|_2^{-1})^q\|v_n\|_{q}^{q}}{(a^{1/2}\|v_n\|_2^{-1})^{2}\|v_n\|_{2^*}^{2}}\\
&\to \frac{(a^{1/2}\|u\|_2^{-1})^2\|\nabla u\|_2^2-\mu
\gamma_q(a^{1/2}\|u\|_2^{-1})^q\|u\|_{q}^{q}}{(a^{1/2}\|u\|_2^{-1})^{2}\|u\|_{2^*}^{2}}=0.
\end{split}
\end{equation*}
Thus $\{u_n\}\subset S_a$ satisfies (\ref{e3.1}).

\textbf{Case 2 ($\mu a^{\frac{q(1-\gamma_q)}{2}}>\bar{a}_N$)}.
Define
\begin{equation*}
f(u):=\frac{\|\nabla
u\|_2^{q\gamma_q}\|u\|_2^{q(1-\gamma_q)}}{\|u\|_q^q}.
\end{equation*}
By (\ref{e3.2}), for any $M>\frac{1}{C_{N,q}^q}$, there exists $u\in
H^1(\mathbb{R}^N)\setminus \{0\}$ such that $f(u)=M$. For any
$\alpha, \beta>0$, we define $\tilde{u}(x):=\alpha u(\beta x)$. By
direct calculations, we have
\begin{equation*}
\|\tilde{u}\|_2=\alpha \beta^{-N/2}\|u\|_2,\ \
\|\tilde{u}\|_q=\alpha \beta^{-N/q}\|u\|_q,\ \  \|\nabla
\tilde{u}\|_2=\alpha \beta \beta^{-N/2}\|\nabla u\|_2
\end{equation*}
and
\begin{equation*}
f(\tilde{u})=\frac{(\alpha \beta \beta^{-N/2}\|\nabla
u\|_2)^{q\gamma_q}(\alpha
\beta^{-N/2}\|u\|_2)^{q(1-\gamma_q)}}{(\alpha
\beta^{-N/q}\|u\|_q)^q}=M.
\end{equation*}
So we can choose
$\alpha=\frac{1}{\|u\|_q}\left(\frac{\|u\|_2}{a^{1/2}\|u\|_q}\right)^{N/2}$
and $\beta=\left(\frac{\|u\|_2}{a^{1/2}\|u\|_q}\right)^{q/2}$ such
that $\|\tilde{u}\|_2^2=a$ and $\|\tilde{u}\|_q=1$.

Under the assumption $\mu a^{\frac{q(1-\gamma_q)}{2}}>\bar{a}_N$, we
have $\mu \gamma_q a^{\frac{q(1-\gamma_q)}{2}}>\frac{1}{C_{N,q}^q}$.
Thus, there exists $\{A_n\}\subset \mathbb{R}$ with $A_n>0$ and
$A_n\to 0$ as $n\to\infty$ such that
\begin{equation*}
M_n:=(\mu \gamma_q+A_n)
a^{\frac{q(1-\gamma_q)}{2}}>\frac{1}{C_{N,q}^q}.
\end{equation*}
For such chosen $M_n$, we choose $\{u_n\}\subset H^1(\mathbb{R}^N)$
such that $\|u_n\|_2^2=a$, $\|u_n\|_q=1$ and $f(u_n)=M_n$. Then we
have
\begin{equation*}
\|\nabla u_n\|_2^2=(\mu \gamma_q+A_n)\|u_n\|_q^q>\mu \gamma_q
\|u_n\|_q^q,
\end{equation*}
\begin{equation*}
1=\|u_n\|_q\leq
\|u_n\|_2^{1-\theta}\|u_n\|_{2^*}^{\theta}=a^{\frac{1-\theta}{2}}\|u_n\|_{2^*}^{\theta}\
\mathrm{with}\ \frac{1}{q}=\frac{1-\theta}{2}+\frac{\theta}{2^*},
\end{equation*}
and
\begin{equation*}
\frac{\|\nabla u_n\|_2^2-\mu
\gamma_q\|u_n\|_q^q}{\|u_n\|_{2^*}^{2}}=\frac{A_n\|u_n\|_q^q}{\|u_n\|_{2^*}^2}\to
0\ \mathrm{as}\ n\to \infty.
\end{equation*}
That is, $\{u_n\}\subset S_a$ satisfies (\ref{e3.1}). The proof is
complete.
\end{proof}

\textbf{Proof of Theorem \ref{thm1.9}}. In view of Lemma
\ref{lem3.2}, the proof of Theorem \ref{thm1.9} is the same to
Proposition 3.2 (2) in \cite{Wei-Wu 2021}.

\bigskip

\textbf{Acknowledgements.} This work is supported by the National
Natural Science Foundation of China (No. 12001403).



\end{document}